\documentclass[11pt]{amsart}
\usepackage{latexsym}
\usepackage{amssymb}
\usepackage{amsmath}

\usepackage{amsfonts}

\newtheorem{theorem}{Theorem}[section]
\newtheorem{lemma}[theorem]{Lemma}
\newtheorem{proposition}[theorem]{Proposition}
\newtheorem{corollary}[theorem]{Corollary}
\newtheorem{definition}[theorem]{Definition}

\newcommand\id{\mathop{\rm id}}

\newcommand\Lat{\mathop{\rm Lat}}
\newcommand\Alg{\mathop{\rm Alg}}

\newcommand{\cl}[1]{\mathcal{#1}}
\newcommand{\bb}[1]{\mathbb{#1}}

\begin{document}

\title{Tensor products of subspace lattices and rank one density}

\author{S. Papapanayides and I. G. Todorov}
\date{28 March 2012}

\address{Department of Pure Mathematics, Queen's University Belfast,
Belfast BT7 1NN, United Kingdom}

\email{spapapanayides01@qub.ac.uk}

\address{Department of Pure Mathematics, Queen's University Belfast,
Belfast BT7 1NN, United Kingdom}

\email{i.todorov@qub.ac.uk}

\begin{abstract}
We show that, if $\cl M$ is a subspace lattice with the property
that the rank one subspace of its operator algebra is weak* dense,
and $\cl L$ is a commutative subspace lattice, then $\cl L\otimes\cl
M$ possesses property (p) introduced in \cite{todorovshulman1}. If
$\cl M$ is moreover an atomic Boolean subspace lattice while $\cl L$
is any subspace lattice, we provide a concrete lattice theoretic
description of $\cl L \otimes \cl M$ in terms of projection valued
functions defined on the set of atoms of $\cl M$. As a consequence,
we show that the Lattice Tensor Product Formula holds for $\Alg \cl
M$ and any other reflexive operator algebra and give several further
corollaries of these results.
\end{abstract}

\maketitle

\section{Introduction}\label{s_intro}

Let $\cl A$ and $\cl B$ be unital operator algebras acting on
Hilbert space. The Lattice Tensor Product Formula (LTPF) problem
asks if the invariant subspace lattice $\Lat (\cl A\otimes\cl B)$ of
the (weak* spatial) tensor product of $\cl A$ and $\cl B$ is the
tensor product of the invariant subspace lattices $\Lat\cl A$ and
$\Lat \cl B$. The origins of this problem can be found in the Tomita
Commutation Theorem, which asserts that the \lq\lq dual'' statement,
namely the Algebra Tensor Product Formula, holds for the projection
lattices of von Neumann algebras. The LTPF problem is related to the
question of reflexivity for subspace lattices, which asks to decide
whether a given lattice of projections on Hilbert space is the
invariant subspace lattice of some operator algebra (see P. R.
Halmos' pivotal paper \cite{boolHalmos}). Although reflexivity
questions have attracted considerable attention in the literature,
little progress has been made on the LTPF problem since the
initiation of its study in \cite{ltpfintro}. One of the reasons for
this is the lack of useful descriptions of the tensor product of two
subspace lattices which, in its own right, is due to the lack of
compatibility between the lattice operations and the strong operator
topology. It is known, however, that the LTPF problem has an
affirmative answer if $\cl A$ and $\cl B$ are von Neumann algebras
one of which is injective \cite{todorovshulman1}, if $\cl A$ is a
completely distributive CSL algebra, while $\cl B$ is any other
operator algebra \cite{Harrison}, as well as when both $\cl A$ and
$\cl B$ are CSL algebras \cite{todorovir}. Even the special case,
where $\cl A$ consists of the scalar multiples of the identity
operator, is in general open, although several partial results were
obtained in \cite{todorovshulman1}.

Properties related to the subspace generated by the rank one
operators in a given operator algebra $\cl A$ (hereafter referred to
as the rank one subspace of $\cl A$) have been widely studied (see,
e.g. \cite[Chapter 23]{dav-book}). In this paper, we continue the
study of the LTPF problem by considering the case where one of the
algebras has the property that its rank one subspace is weak* dense.
The class of operator algebras with this property is rather large;
it includes as a special case the algebras of all operators leaving
two fixed non-trivial subspaces invariant \cite{decKatavolos},
\cite{Papadakis}, as well as the operator algebras of more general
atomic Boolean subspace lattices \cite{memoirs}.

The paper is organised as follows: in Section \ref{s_p}, we show
that if $\cl M$ is a subspace lattice such that the rank one
subspace of the algebra $\cl A = \Alg \cl M$  is weak* dense in $\cl
A$, then the tensor product of $\cl M$ with the full projection
lattice on an infinite dimensional separable Hilbert space is
reflexive. This establishes the LTPF for the algebras $\cl A$ and
$\bb{C}I$. The result is then extended to lattices of the form $\cl
M\otimes\cl L$, where $\cl L$ is a CSL, thus generalising a
corresponding result proved earlier in \cite{todorovshulman1}. In
Section \ref{s_tabsl}, we restrict our attention to the case where
$\cl M$ is an atomic Boolean subspace lattice (ABSL), and achieve a
convenient description of the tensor product $\cl M\otimes\cl L$,
where $\cl L$ is an arbitrary subspace lattice, showing that it is
isomorphic to the lattice of $\cl L$-valued maps defined on the set
of atoms of $\cl M$. We also show that the property of semistrong
closedness of subspace lattices, introduced and studied in
\cite{todorovshulman1}, is preserved under tensoring with $\cl M$
(see Proposition \ref{p_see} for the complete statement). In Section
\ref{s_ltpf}, we show that if $\cl L$ is any reflexive subspace
lattice then the LTPF holds for the algebras $\cl A$ and $\Alg \cl
L$. Some further consequences of the description of the tensor
product from Section \ref{s_tabsl} are also included in Section
\ref{s_ltpf}. In the next section, we collect some preliminaries and
fix notation.

\section{Preliminaries}\label{s_prel}

Let $H$ be a Hilbert space and $\cl S_H$ be the set of all
closed subspaces of $H$. The set $\cl S_H$ is a complete lattice with respect to the
operations of intersection $\wedge$ and closed linear linear span $\vee$.
Using the bijective correspondence between $\cl S_H$ and the set
$\cl P_H$ of all orthogonal projections on $H$, under which a closed subspace $\cl F$
corresponds to the projection with range $\cl F$,
we transfer the lattice structure of $\cl S_H$ to $\cl P_H$, and denote the
lattice operations on $\cl P_H$ obtained in this way again by
$\wedge$ and $\vee$.
A \emph{subspace lattice} on $H$ is a sublattice $\cl L$ of $\cl P_H$
containing $0$ and $I$ and closed in the strong operator topology.

Let $\cl B(H)$ be the algebra of all bounded linear operators acting
on $H$. If $\cl A\subseteq \cl B(H)$, it is customary to denote by
$\Lat\cl A$ the set of all projections on $H$ whose ranges are
invariant under all operators in $\cl A$. It is easy to show that
$\Lat \cl A$ is a subspace lattice. Conversely, given any set of
projections $\cl L\subseteq \cl P_H$, let $\Alg\cl L$ be the set of
all operators on $H$ leaving invariant each element of $\cl L$. It
is easy to see that $\Alg\cl L$ is a unital subalgebra of $\cl B(H)$
closed in the weak operator topology. A subspace lattice $\cl L$ is
called \emph{reflexive} if $\cl L = \Lat\Alg\cl L$. Similarly, an
operator algebra $\cl A$ is called \emph{reflexive} if $\cl A =
\Alg\Lat \cl A$. A subspace lattice $\cl L$ is called a
\emph{commutative subspace lattice} (or \emph{CSL} for short) if $PQ
= QP$ for all $P,Q\in \cl L$ \cite{arvenson}.

If $H_1$ and $H_2$ are Hilbert spaces and $\cl L_1\subseteq \cl
P_{H_1}$ and $\cl L_2\subseteq \cl P_{H_2}$ are subspace lattices,
we denote by $\cl L_1\otimes\cl L_2$ the subspace lattice generated
by the projections of the form $L_1\otimes L_2$ acting on the
Hilbert space tensor product $H_1\otimes H_2$, where $L_1\in \cl
L_1$ and $L_2\in \cl L_2$. Given operator algebras $\cl A_1\subseteq
\cl B(H_1)$ and $\cl A_2\subseteq \cl B(H_2)$, we let $\cl
A_1\otimes\cl A_2$ be the weak* closed operator subalgebra of $\cl
B(H_1\otimes H_2)$ generated by the elementary tensors $A_1\otimes
A_2$, with $A_1\in \cl A_1$ and $A_2\in \cl A_2$. We denote by $I$
the identity operator acting on a separable infinite dimensional
Hilbert space, and set $1\otimes \cl A = \bb{C}I\otimes\cl A$, where
$\bb{C}I = \{\lambda I : \lambda\in \bb{C}\}$. We say that the
Lattice Tensor Product Formula (LTPF) holds for $\cl A_1$ and $\cl
A_2$ if
$$\Lat (\cl A_1\otimes\cl A_2) = \Lat\cl A_1\otimes \Lat \cl A_2.$$
Similarly, the Algebra Tensor Product Formula (ATPF) is said to hold for the subspace lattices
$\cl L_1$ and $\cl L_2$ if
$$\Alg (\cl L_1\otimes\cl L_2) = \Alg\cl L_1\otimes\Alg\cl L_2.$$

The following notion will play an essential role in this paper.

\begin{definition}[\cite{todorovshulman1}]\label{d_p}
A subspace lattice $\cl L$ is said to possess \emph{property (p)} if
the lattice $\cl P_{\ell^2} \otimes \cl L$ is reflexive.
\end{definition}
\noindent It follows from \cite[Proposition 4.2]{todorovshulman1}
that $\cl L$ possesses property (p) if and only if $\cl P_{\ell^2}
\otimes \cl L = \Lat(1 \otimes \Alg\cl L)$.

If $x,y\in H$, we denote by $R_{x,y}$ the rank one operator on $H$ given by
$R_{x,y}(z) = (z,y)x$, $z\in H$. It was shown in \cite{Stronglyreflat} that
the rank one operator $R_{x,y}$ belongs to $\Alg\cl L$ if and only if
there exists $L\in \cl L$ such that $x = Lx$ and $L_- y = 0$, where
$$L_- = \vee\{P\in \cl L : L\not\leq P\}.$$
We say that a subspace lattice $\cl L$ possesses the \emph{rank one
density property} if the subspace of $\Alg\cl L$ generated by the
rank one operators contained in $\Alg\cl L$ is weak* dense in
$\Alg\cl L$. It was shown in \cite{unknown} that if $\cl L$
possesses the rank one density property then it is completely
distributive.

An atomic Boolean subspace lattice (ABSL) is a distributive and
complemented subspace lattice for which there exists a set $\cl E =
\{E_j\}_{j\in J}\subseteq \cl L$ of minimal projections (called
\emph{atoms}) such that for every $L\in \cl L$ there exists
$J_L\subseteq J$ with $L = \vee_{j\in J_E} E_j$ \cite{boolHalmos},
\cite{memoirs}. A special case of interest arises when $\cl E$ has
two elements, see \cite{decKatavolos} and \cite{Papadakis}.

Along with the strong operator topology, we will also use the
semi-strong convergence introduced in \cite{limsupsHalmos}. Namely,
a sequence $(P_n)_{n\in \bb{N}}$ of projections acting on a Hilbert
space $H$ is said to converge semistrongly to a projection $P$ on
$H$, if (a) for every $x\in PH$ there exists a sequence $(x_n)_{n\in
\bb{N}}\subseteq H$ with $x_n \in P_n H$, $n\in \bb{N}$, such that
$x_n\rightarrow_{n\rightarrow\infty} x$, and (b) if $(x_{k})_{k\in
\bb{N}}\subseteq H$ is a convergent sequence of vectors such that
$x_{k}\in P_{n_k}H$, for some increasing sequence $(n_k)_{k\in
\bb{N}}\subseteq \bb{N}$, then $\lim_{k \rightarrow \infty} x_k \in
PH$. It was shown in \cite{limsupsHalmos} that
$P_n\rightarrow_{n\rightarrow \infty} P$ in the strong operator
topology if and only if $P_n\rightarrow_{n\rightarrow \infty} P$
semistrongly and $P_n^{\perp}\rightarrow_{n\rightarrow \infty}
P^{\perp}$ semistrongly, where, for a projection $Q$, we let
$Q^{\perp} = I - Q$ be its orthogonal complement. The weak operator
(resp. strong operator, weak*) topology will be denoted by {\it w}
(resp. {\it s}, {\it w*}).

\section{Property (p)}\label{s_p}

Let $H$ and $K$ be Hilbert spaces, with $K$  infinite dimensional
and separable, let $\cl P = \cl P_K$ be the full projection lattice
on $K$ and let $\cl L\subseteq \cl P$ be a subspace lattice. For any
subset $\cl E\subseteq \cl P_H$, we let $m(\cl E,\cl L)$ be the set
of all maps from $\cl E$ to $\cl L$. If $f,g \in m(\cl E,\cl L)$, we
define $f \vee g$ and $f \wedge g$ to be the elements of $m(\cl
E,\cl L)$ given by
$$(f \vee g)(E) = f(E) \vee g(E) \text{ and } (f \wedge g)(E) = f(E) \wedge
g(E), \ \ \  E \in \cl E.$$ It is clear that, under these
operations, $m(\cl E,\cl L)$ is a complete lattice. Let $\phi_{\cl
E,\cl L} : \cl P_{K\otimes H} \rightarrow m(\cl E,\cl L)$ be the map
sending a projection $Q$ on $K\otimes H$ to the map $f_Q$ given by
\begin{equation}\label{eq_uli}
f_Q(E) = \vee \{P\in \cl L : P \otimes E \leq Q\}, \ \ \ \  E \in \cl E.
\end{equation}
We note that if $E_1,E_2\in \cl E$ are such that $E_1\wedge E_2\in \cl E$, then
\begin{equation}\label{eq_ul}
f_Q(E_1) \vee f_Q(E_2) \leq f_Q(E_1\wedge E_2).
\end{equation}
Dually, let $\theta : m(\cl E,\cl L)\rightarrow \cl P_{K\otimes H}$ be
the map given by
$$\theta(f) = \vee \{f(E) \otimes E : E \in \cl E\}, \ \ \ \ f\in m(\cl E,\cl L).$$

For the rest of this section, fix a
subspace lattice $\cl M\subseteq \cl P_H$ and let $\cl A = \Alg \cl M\subseteq \cl B(H)$.
It is clear that the map $\theta$ sends $m(\cl M,\cl L)$ into $\cl L \otimes \cl M$,
for every subspace lattice $\cl L\subseteq \cl P$.

We first note that $\theta$ is $\vee$-preserving; the proof is
straightforward and we omit it.

\begin{proposition}\label{ultralemma2}
If $(f_{\alpha})_{\alpha\in \bb{A}} \in m(\cl M,\cl L)$ then $\theta(\vee_{\alpha\in \bb{A}} f_{\alpha}) = \vee_{\alpha\in \bb{A}} \theta(f_{\alpha})$.
\end{proposition}

\begin{lemma}\label{l_cyclic}
Let $\cl M\subseteq \cl P_H$ be a subspace lattice with the rank one density property,
$\cl A = \Alg \cl M$ and $\xi \in K\otimes H$. There exists $f\in m(\cl M,\cl P)$ such that
the projection onto the cyclic subspace
$\overline{(1\otimes\cl A)\xi}$ coincides with $\theta(f)$.
\end{lemma}
\begin{proof}
Let $\xi = {\sum}_{j=1}^{\infty} e_j \otimes x_j$, where $(e_j)_{j \in
\mathbb{N}}$ is an orthonormal basis of $K$ and $(x_j)_{j \in
\mathbb{N}}$ is a square-summable sequence in $H$.
Let $f\in m(\cl M,\cl P)$ be the mapping which sends
the projection $L \in \cl M$ to the projection $f(L)$ onto the subspace
$$\overline{\left\{ \sum_{j=1}^{\infty} (x_j,q) e_j : q\in H, L_{-}q = 0 \right\}}.$$

Let $\cl R$ be the rank one subspace of $\cl A$ and
$F\in \cl R$. By \cite{Stronglyreflat},
there exists $m\in \bb{N}$, pairwise distinct projections
$L_i \in \mathcal{M}$, $i=1,\dots ,m$,
and vectors $p^{(i)}_{k} = L_i p^{(i)}_{k}$ and
$q^{(i)}_k = {(L_i)}_{-}^{\perp} q^{(i)}_k$,
$k=1, \dots, l_i$, $l_i\in \bb{N}$, such that $F = {\sum}_{i=1}^m
({\sum}_{k=1}^{l_i} R_{p^{(i)}_{k},q^{(i)}_k})$.
We have
$$(I \otimes F) (\xi) = \sum_{i=1}^{m} \left(\sum_{k=1}^{l_i}
\left(\left(I \otimes R_{p^{(i)}_{k},q^{(i)}_k}\right)
\left(\sum_{j=1}^{\infty} e_j \otimes x_j\right)\right)\right)$$
$$= \sum_{i=1}^m \left(\sum_{k=1}^{l_i}\left( \sum_{j=1}^{\infty} (x_j, q^{(i)}_{k}) e_j \right) \otimes
p^{(i)}_{k}\right).$$ It follows that $(I \otimes F) (\xi) \in
\theta (f)(K \otimes H)$ and since $F$ is an arbitrary element of
$\cl R$, we have that $(1 \otimes \mathcal{R}) \xi \subseteq  \theta
(f)(K \otimes H)$. The property $\overline{\cl R}^{w^*} = \cl A$
easily implies that $\overline{1\otimes \cl R}^{w} = 1\otimes \cl
A$. A standard application of Hahn-Banach's Theorem shows that
$\overline{(1\otimes\cl A)\xi} = \overline{(1\otimes\cl R)\xi}$.
Thus, $\overline{(1 \otimes \mathcal{A}) \xi} \subseteq \theta (f)(K
\otimes H)$.

On the other hand,
if $L \in \mathcal{M}$, $p \in LH$ and $q \in (L_{-})^{\perp} H$, then
$$\left({\sum}_{j=1}^{\infty} (x_j, q) e_j\right) \otimes p
= (I \otimes R_{p,q}) \xi \in \overline{(1 \otimes \mathcal{A})
\xi}.$$ Hence, $(f(L) \otimes L)(K \otimes H) \subseteq \overline{(1
\otimes \mathcal{A}) \xi}$ and so we have that
$\theta (f)(K \otimes H) \subseteq \overline{(1 \otimes \mathcal{A}) \xi}$; thus,
$\overline{(1 \otimes \mathcal{A}) \xi} = \theta (f)(K \otimes H)$
and the proof is complete.
\end{proof}

\begin{theorem}\label{A}
Let $\cl M$ be a subspace lattice with the rank one
density property.
The restriction of the map $\phi = \phi_{\cl M,\cl P}$ to
$\Lat(1\otimes \cl A)$ is injective, $\wedge$-preserving and
\begin{equation}\label{eq_lat1}
\theta \circ \phi|_{\Lat(1\otimes \cl A)} = \id|_{\Lat(1\otimes \cl A)}.
\end{equation}
In particular, $\cl M$ has property (p) and every element of
$\cl P\otimes\cl M$ has the form $\vee_{M\in \cl M} f(M)\otimes M$,
for some map $f : \cl M\rightarrow \cl P$.
\end{theorem}
\begin{proof}
Let $Q \in \Lat(1 \otimes \mathcal{A})$ and $P_L = f_Q(L)$, $L \in
\mathcal{M}$ (see (\ref{eq_uli}) for the definition of the map
$f_Q$). Obviously $$N \overset{def}{=} \theta(\phi(Q))= \underset{L
\in \mathcal{M}}{\vee}(P_L \otimes L) \leq Q.$$ Assume, by way of
contradiction, that there exists $\xi \in Q (K \otimes H) \backslash
N (K \otimes H)$. By Lemma \ref{l_cyclic}, there exists $f\in m(\cl
M,\cl P)$ such that $\overline{(1 \otimes \mathcal{A}) \xi} = \theta
(f)(K \otimes H)$. There exists $M \in \mathcal{M}$ such that $f(M)
\nleq P_M$, for otherwise we would have that $\xi \in N (K \otimes
H)$. Thus,
$$\underset{L \in \mathcal{M}}{\vee}((f(L) \vee P_L) \otimes L) =
(\underset{L \in \mathcal{M}}{\vee}(P_L \otimes L)) \vee
(\underset{L \in \mathcal{M}}{\vee}(f(L) \otimes L))\leq Q;$$ in
particular we have that $(f(M) \vee P_M) \otimes M \leq Q$,
contradicting the maximality of $P_M$. This proves that $Q = N = \theta (\phi(Q))$.

Since the range of $\theta$ is contained in $\cl P\otimes \cl M$,
(\ref{eq_lat1}) implies that
$\Lat (1 \otimes \cl A) \subseteq \mathcal{P} \otimes \cl M$.
Since the converse inclusion is trivial, we conclude
that $\Lat (1 \otimes \cl A) = \mathcal{P} \otimes \cl M$, that is, that
$\cl M$ has property (p).

We next observe that
if $E_1, E_2 \in \mathcal{P} \otimes \mathcal{M}$ then
\begin{equation}\label{eq_claim}
E_1 \leq E_2 \ \Longleftrightarrow \ \phi(E_1) \leq \phi (E_2).
\end{equation}
Indeed, if $\phi(E_1) \leq \phi (E_2)$ then, by (\ref{eq_lat1}),
$$E_1 = \theta(\phi(E_1)) \leq \theta(\phi (E_2)) = E_2.$$
The converse direction follows directly from the definition of $\phi$,
and (\ref{eq_claim}) is proved.

It follows from (\ref{eq_claim}) that $\phi|_{\cl P \otimes\cl M}$ is injective.
It remains to show that $\phi$ is $\wedge$-preserving.
Let $\{E_j\}_{j\in J} \subseteq \cl P \otimes\cl M$,
$f_j = \phi(E_j), j \in J$, and $f = \phi(\underset{j \in
J}{\wedge}E_j)$. By (\ref{eq_claim}) and the fact that
$\underset{i \in J}{\wedge}E_i \leq E_j$ for all $j \in J$, we have
that $f \leq f_j$ for all $j\in J$. Thus,
\begin{equation}\label{eq_fle}
f \leq \underset{i \in J}{\wedge} f_i.
\end{equation}
Now let $g = \phi(\theta(\wedge_{i\in J} f_i))$.
By the definition of $\phi$, we have that $\wedge_{i\in J} f_i \leq g$.
On the other hand, for every $j\in J$, we have by (\ref{eq_lat1}) that
$$\theta(\wedge_{i\in J} f_i) \leq \theta(f_j) = \theta(\phi(E_j)) = E_j.$$
Hence, $\theta(\wedge_{i\in J} f_i) \leq \wedge_{i\in J} E_i$.
By (\ref{eq_claim}), $g \leq f$ and hence $\wedge_{i\in J} f_i \leq f$;
now (\ref{eq_fle}) implies that
$\wedge_{i\in J} f_i = f$, showing that $\phi$ is $\wedge$-preserving.
\end{proof}

\noindent {\bf Remarks (i) } In Theorem \ref{A}, the assumption that
$\cl M$ have the rank one density property is essential. Indeed, let
$\cl D_0$ (resp. $\cl D$) be the multiplication masa of
$L^{\infty}(0,1)$ (resp. $L^{\infty}([0,1]^2)$) acting on $L^2(0,1)$
(resp. $L^2(0,1)\otimes L^2(0,1)$), and let $\cl N_0$ (resp. $\cl
N$) be the projection lattice of $\cl D_0$ (resp. $\cl D$). We have
that $\cl N \equiv \cl N_0 \otimes \cl N_0 \subseteq \cl P \otimes
\cl N_0$. For a measurable subset $\gamma$ of $(0,1)$ or of
$(0,1)^2$, we write $M_{\gamma}$ for the projection of
multiplication by the characteristic function of $\gamma$.

Let $C$ be a non-null Cantor
subset of $[0,1)$ and equip $[0,1)$ with the group operation of
addition$\mod1$. Set
$$F = \{(x,y) \in [0,1) \times [0,1) : x-y \in  C\}.$$
The set $F$ is clearly non-null; we claim that it does not contain
any non-trivial measurable rectangles. Indeed, suppose, by way of
contradiction, that there exist non-null measurable subsets $\alpha$
and $\beta$ of $[0,1)$, such that $\alpha \times \beta \subseteq F$.
It follows by the definition of $F$ that $\alpha - \beta = \{a-b : a
\in \alpha, b \in \beta\}$ is contained in $C$. By a well-known
version of Steinhaus' Theorem, we have that $\alpha - \beta$
contains an open interval. However, $\alpha - \beta \subseteq C$ and
$C$ has empty interior, a contradiction. Thus, $F$ does not contain
any non-trivial measurable rectangles and hence there exist no
non-null subsets $\alpha$ and $\beta$ of $[0,1)$ such that
$M_{\alpha} \otimes M_{\beta} \leq M_F$.

We will prove that $\phi(M_F)(M_{\beta}) = 0$ for every measurable
$\beta$. Fix such a $\beta$ and set $P = \phi(M_F)(M_{\beta})$. Let
$P_1$ be the projection onto $\overline{\cl D_0 PK}$; then $P_1\in
\cl N_0$ and $P\leq P_1$. Since $P_1\otimes M_{\beta}\leq M_F$, we
have that $P_1 = 0$, showing that $P = 0$. It follows that identity
(\ref{eq_lat1}) from Theorem \ref{A} does not hold in the case $\cl
M = \cl N_0$.

\medskip

{\bf (ii) }
Let $\cl M$ be an ABSL with the rank one density property,
and $E_1$ and $E_2$ be atoms of $\cl M$. Also let $L_i \in
\cl P$, $i=1,2$, be such that $L_1 \wedge L_2 \neq 0$ and $M = (L_1
\otimes E_1) \vee (L_2 \otimes E_2)$.
Clearly
$$M = (L_1 \otimes E_1) \vee (L_2 \otimes E_2) \vee ((L_1 \wedge
L_2) \otimes (E_1 \vee E_2))$$ and thus the representation in
Theorem \ref{A} is not unique.
%If $f_i \in m(\cl M, \cl
%P)$, $i=1,2$ such that $f_i(E_j) = L_j$, $i, j = 1, 2$, $f_1 (P) =
%0$ for all $P \neq E_i$, $i=1,2$, $f_2(E_1 \vee E_2) = L_1 \wedge
%L_2$ and $f_2(P) = 0$ for all $P \neq E_1, E_2, E_1 \vee E_2$, then
%$\theta(f_1) = M = \theta(f_2)$ and thus $\theta|_{m(\cl M,\cl P)}$ is not one-to-one
%and $\phi |_{\cl P \otimes \cl M} \circ \theta \neq \id$.

The map $\phi|_{\cl P \otimes \cl M}$ is not a $\vee$-preserving
and thus not a lattice homomorphism. Indeed, it is easy to check that $\phi
(L_i \otimes E_i)(E_i) = L_i$, $i= 1,2$, and $\phi (L_i \otimes
E_i)(E_1\vee E_2) = 0$. Also, $\phi (M)(E_1 \vee E_2) \geq L_1 \wedge L_2 >
0$.
Thus,
$$\phi (M)(E_1 \vee E_2) \neq 0 = (\phi (L_1 \otimes
E_1)(E_1 \vee E_2)) \vee (\phi (L_2 \otimes E_2)(E_1 \vee E_2))$$
and hence $\phi (M) \neq (\phi (L_1 \otimes E_1)) \vee (\phi (L_2
\otimes E_2))$.

\bigskip

Theorem \ref{A} can now be extended as follows.

\begin{theorem}\label{4elcsl3}
Let $\mathcal{L}$ be a separably acting CSL and $\mathcal{M}$ be a subspace lattice
with the rank one density property. Then $\mathcal{L}
\otimes \mathcal{M}$ possesses property (p).
\end{theorem}
\begin{proof}
By Theorem \ref{A}, $\mathcal{P}
\otimes \mathcal{M}$ is reflexive.
If $\mathcal{L}$ is a finite CSL then $\cl L$ is totally
atomic and \cite[Theorem 12, Corollary 2]{Harrison} imply that
$\mathcal{L} \otimes \mathcal{P} \otimes \mathcal{M}$ is reflexive.
Hence, $\mathcal{L} \otimes \mathcal{M}$ has property (p).

Now let $\mathcal{L}$ be an arbitrary separably acting CSL, $\{L_i\}_{i \in
\bb{N}}$ be a strongly dense subset of $\cl L$, and $\cl L_n$ be the
subspace lattice generated by the set $\{L_i\}_{i=1}^n$, $n\in
\bb{N}$; clearly, $\cl L_n$ is finite for all $n \in \bb{N}$.
Since $\cl L = \overline{\underset{n \in \bb{N}}{\cup}\cl L_n}^s$, we have that
$\mathcal{L} \otimes \mathcal{M} = \underset{n \in \mathbb{N}}{\vee} ({\mathcal{L}}_n \otimes
\mathcal{M})$. By the previous paragraph, ${\mathcal{L}}_n \otimes
\mathcal{M}$ has property (p) for all $n \in \mathbb{N}$. By the
strict approximativity of property (p) (see \cite[Proposition
4.1]{todorovshulman2}), $\cl L \otimes \cl M$ has property (p).
\end{proof}

\section{Tensoring with atomic Boolean subspace lattices}
\label{s_tabsl}

In this section, we restrict our attention to the case where $\cl M$
is an Atomic Boolean Subspace Lattice (ABSL) possessing the ultraweak
rank one density property. Two atom ABSLs, namely lattices of the
form $\{0,P,Q,I\}$, where $P\wedge Q = 0$ and $P\vee Q = I$, satisfy
this property \cite{Papadakis} and it is not difficult to show that
the rank one density property is preserved under taking
meshed product (see \cite{memoirs} for the definition and properties
of this construction).

Our aim is to show that, if $\cl M$ is an ABSL with the rank one density
property, $\cl E$ is the set of its atoms, and $\cl L$ is an arbitrary
subspace lattice, then the map $\theta$ is an isomorphism from
$m(\cl E,\cl L)$ onto $\cl L\otimes\cl M$.
We first establish an important special case.

\begin{lemma}\label{4elgen02}
Let $\mathcal{M}$ be an ABSL acting on a Hilbert space $H$ having
the rank one density property and let $\cl E = \{E_j : j
\in J\}$ be the set of its atoms. Then $\theta|_{m(\cl E,\cl
P)}$ is a complete lattice isomorphism of $m(\cl E,\cl P)$ onto $\cl
P\otimes\cl M$ with inverse $\phi_{\cl E,\cl P}$.
\end{lemma}
\begin{proof}
Let ${\mathcal{M}}_j = \{L \in \mathcal{M} : E_j \leq L\}$, $j\in J$.
Fix $M \in \cl P \otimes \cl M$ and let $f = \phi_{\cl M, \cl P}(M)$. By Theorem \ref{A},
\begin{eqnarray*}%\label{zazaza}
\nonumber M &=& \underset{L \in \mathcal{M}}{\vee}(f(L) \otimes L) =
\underset{L \in \mathcal{M}}{\vee}(f(L) \otimes (\underset{E_j \leq L}{\vee} E_j)) \\
\nonumber &=& \underset{L \in \mathcal{M}}{\vee}(\underset{E_j \leq L}{\vee}(f(L) \otimes
E_j)) = \underset{j \in J}{\vee}((\underset{L \in
{\mathcal{M}}_j}{\vee} f(L)) \otimes E_j)\\
& = &
\underset{j \in J}{\vee} f(E_j) \otimes E_j,
\end{eqnarray*}
where the last identity follows from (\ref{eq_ul}).
Thus,
\begin{equation}\label{eq_rem}
(\theta\circ \phi_{\cl E,\cl P})(M) = M, \ \ \ M\in \cl P\otimes\cl M.
\end{equation}
Let $\phi = \phi_{\cl E, \cl P}|_{\cl P\otimes\cl M}$ for brevity.
We next check that
\begin{equation}\label{eq_phith}
(\phi \circ \theta)(f) = f, \ \ \ f\in m(\cl E,\cl P).
\end{equation}
Let $f\in m(\cl E,\cl P)$, $g = \phi \circ \theta (f)$, $M_j = f(E_j)$ and
$P_j = g(E_j)$, $j\in J$. Set $M = \theta(f)$; by (\ref{eq_rem}),
\begin{equation}\label{eq_nn}
M = \theta(\phi(M)) = \theta(g).
\end{equation}
By the definition of $\phi$, we have that
$f\leq g$, that is, $M_j \leq P_j$ for all $j\in J$.
Suppose that there exists $i \in J$ such that $M_i < P_i$. We have that
\begin{eqnarray*}
M &=& \underset{j \in J}{\vee} (M_j \otimes E_j) \leq (M_i \otimes
E_i) \vee
\left({\underset{j \neq i}{\bigvee}} (I \otimes E_j)\right) \\
  &=& (M_i \otimes E_i) \vee
\left(I \otimes \left({\underset{j \neq i}{\bigvee}} E_j\right)\right) \\
  &=& (M_i \otimes E_i) \vee \left(M_i
\otimes \left({\underset{j \neq i}{\bigvee}}
E_j\right)\right) \vee \left({M_i}^{\perp} \otimes \left({\underset{j \neq i}{\bigvee}} E_j\right)\right)\\
&=& (M_i \otimes I) \oplus \left({M_i}^{\perp} \otimes
\left({\underset{j \neq i}{\bigvee}}
E_j\right)\right),
\end{eqnarray*}
where for last equality we have used the fact that $M_i \otimes I$
and ${M_i}^{\perp} \otimes \left({\underset{j \neq i}{\bigvee}}
E_j\right)$ are orthogonal. Let now $0 \neq p \in (P_iK) \ominus
(M_iK)$ and $0 \neq e \in E_i H$. Using (\ref{eq_nn}), we have that
$p \otimes e \in (P_i \otimes E_i)(K \otimes H) \subseteq M(K
\otimes H)$ and that $(M_i \otimes I)(p \otimes e) = 0$. Hence
$$p \otimes e \in \left({M_i}^{\perp} \otimes \left({\underset{j \neq i}{\bigvee}} E_j\right)\right)(K \otimes H)$$
and therefore
$$0 \neq e \in \left(\left({\underset{j \neq i}{\bigvee}} E_j\right)H\right) \wedge (E_iH) = \{0\},$$
a contradiction. Hence $f = g = \phi(\theta (f))$ and (\ref{eq_phith}) is proved.

By Proposition \ref{ultralemma2}, $\theta|_{m(\cl E,\cl P)}$ is
$\vee$-preserving. By Theorem \ref{A}, $\phi$ is
$\wedge$-preserving. Let $(f_{\alpha})_{\alpha\in \bb{A}}\subseteq
m(\cl E,\cl P)$. Using (\ref{eq_phith}), we have
$$\phi(\theta(\wedge_{\alpha\in \bb{A}} f_{\alpha})) = \wedge_{\alpha\in \bb{A}} f_{\alpha}
= \wedge_{\alpha\in \bb{A}} (\phi\circ\theta)(f_{\alpha}) =
\phi(\wedge_{\alpha\in \bb{A}}\theta(f_{\alpha})).$$
By (\ref{eq_rem}), $\phi$ is injective and so $\theta(\wedge_{\alpha\in \bb{A}} f_{\alpha})
= \wedge_{\alpha\in \bb{A}}\theta(f_{\alpha})$.
The proof is complete.
\end{proof}

It will be helpful to isolate the following statement contained in
Lemma \ref{4elgen02} for future reference.

\begin{corollary}\label{c_cup}
Let $\mathcal{M}$ be an ABSL acting on a Hilbert space $H$ having
the rank one density property and let $\cl E = \{E_j : j
\in J\}$ be the set of its atoms.
If $M \in \mathcal{P} \otimes \mathcal{M}$, then there exists a
unique family $(P_j)_{j \in J} \subseteq \mathcal{P}$ such that $M
= \underset{j \in J}{\vee} (P_j \otimes E_j)$.
\end{corollary}

\begin{lemma}\label{ssop1}
Let $H$ be a Hilbert space, $\cl M$ be an ABSL on $H$ with atoms
$E_j$, $j \in J$ having the rank one density property, and $\{f, f_n
: n\in \bb{N}\}\subseteq m(\cl E,\cl P)$.

(i) \ If $\theta(f_n)\rightarrow_{n\rightarrow\infty} \theta(f)$ semistrongly then $f_n(E_j)\rightarrow_{n\rightarrow\infty} f(E_j)$
semi-strongly for every $j\in J$.

(ii) If $f_n(E_j)\rightarrow_{n\rightarrow\infty} f(E_j)$ semistrongly for every $j\in J$ then there exists a
subsequence $(\theta(f_{n_k}))_{k\in \bb{N}}$ of $(\theta(f_n))_{n\in \bb{N}}$ such that
$\theta(f_{n_k}) \rightarrow_{k\rightarrow\infty} \theta(f)$ semistrongly.
\end{lemma}
\begin{proof}
Let $L_n^j = f_n(E_j)$ and $L_j = f(E_j)$, $j\in J$, $n\in \bb{N}$.

(i) Fix $k \in J$ and let $(x_i)_{i \in \mathbb{N}}$ be a sequence
such that $x_i \in L_{n_i}^kK$, $i \in \mathbb{N}$, and $x_i
\rightarrow x$ (where the sequence $(n_i)_{i \in \bb N}\subseteq
\bb{N}$ is strictly increasing). Fix a non-zero vector $p \in E_kH$.
It follows that $x_i \otimes p \rightarrow x \otimes p$. Clearly,
$x_i \otimes p \in \theta(f_{n_i})(K \otimes H) $ for all $i \in
\mathbb{N}$ and thus, by hypothesis, $x \otimes p \in \theta(f)(K
\otimes H)$. Let
$$\cl W = \{y : y \otimes p \in \theta(f)(K\otimes H) \text{ for all } p \in E_kH\}.$$
Clearly, $\cl W$ is a closed subspace such that
$L_k K\subseteq \cl W$ and $x \in \cl W$.
Also, $\cl W \otimes E_k H \subseteq \theta(f)(K \otimes H)$.
By Lemma \ref{4elgen02},
$$\cl W \otimes E_kH \subseteq ((\underset{j \in
J}{\vee} (L_j \otimes E_j))(K \otimes H)) \wedge (K \otimes E_kH) =
L_kK \otimes E_kH.$$
It follows that $\cl W\subseteq L_kK$ and so $x\in L_kK$.

Let $q$ be a non-zero vector in $H$ such that $(\underset{j \neq
k}{\vee}E_j) q = (E_k)_{-}q = 0$. Write $q = p_0 + p_0'$ where $p_0
= E_k p_0$ and $E_k p_0' = 0$. Since $E_k^{\perp} \wedge
(\underset{j \neq k}{\vee}E_j)^{\perp} = (E_k \vee (\underset{j \neq
k}{\vee}E_j))^{\perp} = 0$, it follows that $p_0 \neq 0$ and thus
$(p_0,q) \neq 0$. Let $p = \frac{p_0}{(p_0,q)}$; we have that
$R_{p,q} \in \Alg \cl M$. Clearly $R_{p,q} p = p$ and $R_{p,q}$
annihilates $\underset{j \neq k}{\vee}E_j$. Fix $x \in L_kK$. By
hypothesis, there exist a sequence $(\xi_n)_{n \in \mathbb{N}}$ such
that $\xi_n = \theta(f_n) \xi_n$, $n\in \bb{N}$, and $\xi_n
\rightarrow_{n\rightarrow\infty} x\otimes p$. Thus,
$$x\otimes p = (I\otimes R_{p,q})(x\otimes p) = \lim_{n\rightarrow \infty} (I\otimes R_{p,q})\xi_n.$$
By the definition of $R_{p,q}$, we have that $(I\otimes
R_{p,q})\xi_n \in (L_n^k\otimes E_k)(K\otimes H)$. Let $\psi : K
\otimes H\rightarrow K$ be the bounded linear operator such that
$$\psi(x_1 \otimes x_2) = \frac{(x_2,p)}{\|p\|^2}x_1, \text{ } x_1
\in K, \  x_2 \in H.$$ Clearly, $\psi (I \otimes R_{p,q}) \xi_n \in
L_n^k K$ for all $n \in \mathbb{N}$, and $\psi((I \otimes
R_{p,q})\xi_n)\rightarrow \psi(x\otimes p) = x$. This shows that
$L_n^k \rightarrow L_k$ semistrongly.

\smallskip

(ii) Suppose that $f_n(E_j)\rightarrow f(E_j)$ semistrongly for all $j\in J$.
By the weak compactness of the unit ball of
$\cl B(H)$ (see, e.g. \cite[Proposition 5.5]{conway}), there exists a subsequence
$(\theta(f_{n_k}))_{k\in \bb{N}}$ of $(\theta(f_n))_{n\in \bb{N}}$
and a positive contraction $W$ on $H$ such that
$\theta(f_{n_k}) \rightarrow_{k\rightarrow\infty} W$ in the weak operator topology.
By \cite{limsupsHalmos}, $(\theta(f_{n_k}))_{k\in \bb{N}}$ converges
semistrongly to the orthogonal projection $Q$ onto $\ker (I-W)$.
By Theorem \ref{A}, $\cl P \otimes \cl M$ is reflexive and, by \cite[Proposition
3.1]{todorovshulman1}, it is semistrongly closed.
Thus, $Q\in \cl P\otimes\cl M$ and, by Lemma \ref{4elgen02},
$Q = \theta(g)$ for some $g\in m(\cl E,\cl P)$.
By (i), $f_{n_k}(E_j)\rightarrow_{k\rightarrow\infty} g(E_j)$ semistrongly.
By the uniqueness of the semistrong
limit, $f(E_j) = g(E_j)$ for all $j \in J$, that is, $f = g$ and so
$\theta(f_{n_k})\rightarrow_{k\rightarrow\infty} \theta(f)$ semistrongly.
\end{proof}

The next proposition is certainly well-known; since we were not able to
find a corresponding reference, we include its short proof for the
convenience of the reader.

\begin{proposition}\label{ultraweakABSL}
Let $\cl M$ be an ABSL, $\cl E = \{E_j\}_{j \in J}$ be the set of
its atoms, and
let $D_j = \wedge_{i \neq j}({E_i}^\perp)$, $j\in J$.
Then $\cl M^\perp \stackrel{def}{=} \{L^\perp : L \in \cl M\}$ is an ABSL whose set of atoms
is $\cl D = \{D_j\}_{j \in J}$.
\end{proposition}
\begin{proof}
It is a direct consequence of the de Morgan laws
that $\cl M^{\perp}$ is distributive and that
if $L \in \cl M$ and $L' \in \cl M$ is the complement of $L$ in $\cl M$, then
${L'}^\perp$ is a complement of $L^{\perp}$ in $\cl M^{\perp}$.

Let $L \in \cl M^\perp$. If $0 \leq L < D_j$ for some
$j \in J$, then $\vee_{i \neq j}E_i = {D_j}^\perp <
L^\perp  \in \cl M$. Since $L^\perp$ is equal to the closed linear
span of the atoms that it majorises, it must contain $E_i$ and hence
$L^\perp = I$, that is, $L = 0$.
Thus, $D_j$ is an atom of $\cl M^{\perp}$, for each $j\in J$.

If $L \in \cl M^\perp$, then there exists $S \subseteq J$ such that
$L^\perp = \vee_{j \in S}E_j$. By distributivity,
$$L = \vee_{j \notin S}(\wedge_{i \neq j}E_i^\perp) = \vee_{j \notin S} D_j.$$
We thus showed that $\cl M^\perp$ is an ABSL with atoms $\{D_j : j \in J\}$.
\end{proof}

%In the notation of Proposition \ref{ultraweakABSL}, let
%$\theta^{\perp} : m(\cl D,\cl P)$ be the map given by
%$\theta^{\perp}(g) = \vee_{j\in J} g(D_j)\otimes D_j$. Also,
In the rest of the section, we adopt the notation from Proposition \ref{ultraweakABSL}.
If $f\in m(\cl E,\cl P)$, let $f^{\perp}\in m(\cl D,\cl P)$ be the map
given by $f^{\perp}(D_j) = f(E_j)^{\perp}$, $j\in J$.

\begin{lemma}\label{l_perpe}
Let $\cl M$ be an ABSL with the rank one density property and
$\cl E$ be the set of its atoms.
If $f\in m(\cl E,\cl P)$ then $\theta(f)^{\perp} = \theta(f^{\perp})$.
\end{lemma}
\begin{proof}
Since $\cl M$ has the rank one density property, the identity
$\Alg \cl M^{\perp} = (\Alg\cl M)^*$ implies that $\cl M^{\perp}$
has the rank one density property as well.
Let $f\in m(\cl E,\cl P)$ and $L_j = f(E_j)$, $j\in J$. Then
\begin{eqnarray*}
\theta(f)^{\perp} & = & (\underset{j \in J}{\vee}(L_j \otimes
E_j))^{\perp} = \underset{j \in J}{\wedge} (L_j \otimes E_j)^\perp \nonumber \\
& = & \underset{j \in J}{\wedge} (({L_j}^{\perp} \otimes I) \vee
(L_j \otimes {E_j}^\perp)) \nonumber \\
                                                    & = & \underset{j \in J}{\wedge} (({L_j}^{\perp} \otimes D_j) \vee
                                                    ({L_j}^{\perp} \otimes {E_j}^\perp) \vee (L_j \otimes {E_j}^\perp)) \nonumber \\
                                                    &=& \underset{j \in J}{\wedge} (({L_j}^{\perp} \otimes D_j)
                                                    \vee (I \otimes {E_j}^\perp))     \nonumber \\
                                                    &=& \underset{j \in J}{\wedge} (({L_j}^{\perp} \otimes D_j)
                                                    \vee (I \otimes (\underset{i \neq j}{\vee}D_i)))     \nonumber \\
                                                    &=& \underset{j \in J}{\wedge} (({L_j}^{\perp} \otimes D_j)
                                                    \vee (\underset{i \neq j}{\vee}(I \otimes
                                                    D_i))
                                                    = \underset{j \in J}{\vee} ({L_j}^{\perp} \otimes D_j) = \theta(f^{\perp}),
\end{eqnarray*}
where at the second last equality we used Lemma \ref{4elgen02}.
\end{proof}

The main result of this section is the following.

\begin{theorem}\label{4elang6}
Let $\mathcal{L}$ be a subspace lattice acting on a Hilbert space
$K$ and $\mathcal{M}$ be an ABSL with the rank one density property.
Let $\cl E = \{E_j: j \in J\}$ be the set of atoms of $\cl M$. Then
$\theta|_{m(\cl E,\cl L)}$ is a complete lattice isomorphism of
$m(\cl E,\cl L)$ onto $\cl L\otimes\cl M$ with inverse $\phi_{\cl
E,\cl P}|_{\cl L\otimes\cl M}$.
\end{theorem}
\begin{proof}
Let
$$\mathcal{F} = \theta(m(\cl E,\cl L)) = \{\underset{j \in J}{\vee} (L_j \otimes E_j) :
L_j \in \cl L, j \in J \}.$$
By Lemma \ref{4elgen02}, $\cl F$ is a projection lattice.
We will show that $\mathcal{F}$ is strongly closed.
Let $f_n\in m(\cl E,\cl L)$, $n\in \bb{N}$, and let $Q$ be a projection with
$\theta(f_n)\rightarrow Q$ in the strong operator topology.
Since $\theta(f_n) \in \cl P\otimes\cl M$ for all $n$, we have that
$Q\in \cl P\otimes\cl M$.
By Lemma \ref{4elgen02}, $Q = \theta(f)$ for some $f\in m(\cl E,\cl P)$.
Since $\theta(f_n)\rightarrow_{n\rightarrow\infty} \theta(f)$ semistrongly \cite{limsupsHalmos},
Lemma \ref{ssop1} (i) implies that $f_n(E_j)\rightarrow_{n\rightarrow\infty} f(E_j)$ semistrongly,
for all $j\in J$.

Since $\cl M$ has the rank one density property, $\cl M^{\perp}$ does so as well.
By \cite{limsupsHalmos},
$\theta(f_n)^{\perp}\rightarrow_{n\rightarrow\infty} \theta(f)^{\perp}$ semistrongly
and by Lemma \ref{l_perpe}, $\theta(f_n^{\perp})\rightarrow_{n\rightarrow\infty} \theta(f^{\perp})$ semistrongly.
By Lemma \ref{ssop1} (i), $f_n^{\perp}(D_j)\rightarrow_{n\rightarrow\infty} f^{\perp}(D_j)$ semistrongly
for all $j\in J$, that is,
$f_n(E_j)^{\perp}\rightarrow_{n\rightarrow\infty} f(E_j)^{\perp}$ semistrongly for all $j\in J$.
By \cite{limsupsHalmos},
$f_n(E_j)\rightarrow_{n\rightarrow\infty} f(E_j)$ in the strong operator topology and, since
$\cl L$ is strongly closed, we conclude that $f(E_j)\in \cl L$ for all $j\in J$.
Thus, $\cl F$ is strongly closed. It follows that $\cl F = \cl L\otimes\cl M$.

If $Q\in \cl L\otimes\cl M$ then by Corollary \ref{c_cup}, there exists a unique
$f\in m(\cl E,\cl P)$ such that $\theta(f) = Q$. Since
$\cl L\otimes\cl M = \theta(m(\cl E,\cl L))$, we have that $f\in m(\cl E,\cl L)$.
Thus, $\phi_{\cl E,\cl P}(Q)\in m(\cl E,\cl L)$, and the rest of
the statements follow from Lemma \ref{4elgen02}.
\end{proof}

We include some immediate corollaries of Theorem \ref{4elang6}.

\begin{corollary}\label{c_cup2}
Let $\mathcal{M}$ be an ABSL acting on a Hilbert space $H$ having
the rank one density property, $\cl E = \{E_j : j \in J\}$ be the
set of its atoms and $\cl L$ be any subspace lattice. If $M \in
\mathcal{L} \otimes \mathcal{M}$, then there exists a unique family
$(P_j)_{j \in J} \subseteq \mathcal{P}$ such that $M = \underset{j
\in J}{\vee} (P_j \otimes E_j)$. Moreover, $P_j\in \cl L$ for each
$j\in J$.
\end{corollary}

\begin{corollary}\label{c_dist}
Let $\mathcal{L}$ be a subspace lattice
and $\mathcal{M}$ be an ABSL with the rank one density property.
If $\cl L$ is distributive then so is $\cl L\otimes\cl M$.
\end{corollary}

We finish this section with a stability result about semistrong closedness.
We refer the reader to \cite{todorovshulman1}, where semistrongly closed
subspace lattices were studied in detail.

\begin{proposition}\label{p_see}
Let $\mathcal{L}$ be a subspace lattice acting on a Hilbert space
$K$ and $\mathcal{M}$ be an ABSL
acting on a Hilbert space $H$, having the
rank one density property.
The lattice $\cl L$ is semistrongly closed if and only if the lattice $\cl L\otimes\cl M$
is semistrongly closed.
\end{proposition}
\begin{proof}
Suppose that $\cl L$ is semistrongly closed and assume that $\{Q_n :
n\in \bb{N}\}\subseteq \cl L\otimes \cl M$ with $Q_n\rightarrow Q$
semistrongly for some projection $Q$ on $K\otimes H$. By Theorem
\ref{A}, $\cl P\otimes \cl M$ is reflexive, and by
\cite{todorovshulman1}, it is semistrongly closed; hence, $Q\in \cl
P\otimes\cl M$. Thus, by Lemma \ref{4elgen02}, $Q = \theta(f)$ for
some $f\in m(\cl E,\cl P)$, where $\cl E$ is the set of atoms of
$\cl M$. By Theorem \ref{4elang6}, there exist $f_n\in m(\cl E,\cl
L)$ such that $Q_n = \theta(f_n)$, $n\in \bb{N}$. By Lemma
\ref{ssop1} (i), $f_n(E_j)\rightarrow f(E_j)$ semistrongly and since
$\cl L$ is semistrongly closed, $f(E_j)\in \cl L$; therefore, $Q\in
\cl L\otimes\cl M$.

Conversely, suppose that $\cl L\otimes\cl M$ is semistrongly closed. Fix an atom $E$ of $\cl M$.
Suppose that $(L_n)_{n\in \bb{N}}\subseteq \cl L$ and that $L_n\rightarrow_{n\rightarrow \infty} L$ semistrongly for some
projection $L\in \cl P$.
By Lemma \ref{ssop1} (ii), there exists a subsequence $(n_k)_{k\in \bb{N}}$
with $L_{n_k}\otimes E\rightarrow_{k\rightarrow\infty} L\otimes E$ semistrongly. Since $\cl L\otimes\cl M$
is semistrongly closed, $L\otimes E\in \cl L\otimes\cl M$ and, by
Corollary \ref{c_cup2}, $L\in \cl L$.
\end{proof}

\section{LTPF and other consequences}\label{s_ltpf}

The next theorem, along with Corollary \ref{4elgen4}, are the main results of this section.
We also give some more consequences of the results from the previous sections.

\begin{theorem}\label{4elgen3minus}
Let $\mathcal{L}$ be a subspace lattice
acting on a Hilbert space $K$ and $\mathcal{M}$ be an ABSL acting on a Hilbert
space $H$ and having the rank one density property. Let $\cl E
= \{E_j : j \in J\}$ be the set of atoms of $\cl M$. Then
\begin{eqnarray}\label{eq_long}
\nonumber \Lat\Alg(\mathcal{L} \otimes \mathcal{M}) & = & (\Lat\Alg
\mathcal{L}) \otimes \mathcal{M}\\ & = & \{ \underset{j \in
J}{\vee}(f(E_j) \otimes E_j) : f \in m(\cl E, \Lat\Alg\mathcal{L})\}.
\end{eqnarray}
\end{theorem}
\begin{proof}
The second equality follows from Corollary \ref{c_cup2}.
By hypothesis, the subalgebra of $\mathcal{A} = \Alg \cl M$
generated by the rank one operators in $\mathcal{A}$ is dense in
$\mathcal{A}$ in the ultraweak topology. Hence, it follows from
\cite[Theorem 2.1 and Proposition 1.1]{Kraus} that $\Alg(\mathcal{L} \otimes \mathcal{M}) =
(\Alg\mathcal{L}) \otimes \mathcal{A}$. Thus,
\begin{eqnarray*}
(\Lat\Alg \mathcal{L}) \otimes \mathcal{M} &=& \Lat\Alg \mathcal{L}
\otimes \Lat\Alg \mathcal{M} \\
&\subseteq& \Lat(\Alg \mathcal{L} \otimes \mathcal{A}) =
\Lat\Alg(\mathcal{L} \otimes \mathcal{M}).
\end{eqnarray*}

It remains to prove the inclusion $\Lat\Alg(\mathcal{L} \otimes
\mathcal{M}) \subseteq (\Lat\Alg \mathcal{L}) \otimes \mathcal{M}$.
Let $k \in J$; then $(E_k)_{-} = \underset{j \neq k}{\vee} E_j$. Fix
$E \in \Lat\Alg(\mathcal{L} \otimes \mathcal{M})$. Using Theorem
\ref{A}, we have
$$\Lat\Alg(\mathcal{L} \otimes \mathcal{M}) \subseteq \Lat\Alg(\mathcal{P} \otimes
\mathcal{M}) = \cl P\otimes\cl M;$$ Corollary \ref{c_cup} now
implies that there are unique projections $P_j \in \mathcal{P}$, $j
\in J$ such that $E = \underset{j \in J}{\vee} (P_j \otimes E_j)$.
The proof will be complete if we show that $P_j \in \Lat\Alg \cl L$
for all $j \in J$. Let $\mathcal{S}$ be the set of all rank one
operators $R_{x,y}$ such that $E_k x = x$ and $(E_k)_- y = 0$.
Clearly, $\cl S \subseteq \mathcal{A}$. Also let $T \in \Alg
\mathcal{L}$ and $0 \neq S \in \mathcal{S}$. It is straightforward
that $T \otimes S$ annihilates $(P_j \otimes E_j)(K \otimes H)$ for
all $j \neq k$ and belongs to $\Alg(\cl L \otimes \cl M)$. Thus,
\begin{eqnarray*}
\overline{T P_kK} \otimes E_kH & = & \underset{S \in \mathcal{S}}{\vee}
(\overline{T P_kK} \otimes \overline{SH}) =
\underset{S \in \mathcal{S}}{\vee}(\overline{T P_kK} \otimes \overline{S E_kH})\\
& = &
\underset{S \in \mathcal{S}}{\vee}\overline{(T \otimes S)E (K \otimes H)} \subseteq E(K\otimes H).
\end{eqnarray*}
Let $x\in P_k K$. For every $y\in E_k H$, we have by the last inclusion that
$Tx \otimes y\in E(K \otimes H)$.
Denoting by $[Tx]$ the projection on the subspace $\{\lambda Tx :
\lambda \in \mathbb{C}\}$, we have that
\begin{eqnarray*}
& & ((P_k \vee [Tx]) \otimes E_k)\vee (\underset{j \neq k}{\vee} P_j \otimes E_j)\\
& = & ([Tx] \otimes E_k) \vee (P_k \otimes E_k) \vee (\underset{j \neq k}{\vee} (P_j \otimes E_j))
\subseteq E = \underset{j \in J}{\vee} (P_j \otimes E_j).
\end{eqnarray*}
By Corollary \ref{c_cup2}, $P_k \vee [Tx] = P_k$ and thus $Tx \in
P_kK$. This shows that $P_k \in \Lat\Alg \mathcal{L}$ and (\ref{eq_long}) is proved.
\end{proof}

\begin{corollary}\label{4elgen4}
Let $K$ be a Hilbert space, $\mathcal{L}$ be a reflexive subspace
lattice acting on $K$ and $\mathcal{M}$ be an ABSL acting on a Hilbert
space $H$, having the rank one density property.
Then the LTPF holds for $\Alg \cl L$ and $\Alg \cl M$.
\end{corollary}
\begin{proof}
The ATPF holds for $\cl L$ and $\cl M$ because $\cl M$ has the
ultraweak rank one density property (see \cite[Theorem 2.1 and Proposition 1.1]{Kraus}).
Let $\cl A = \Alg \cl M$ and $\cl B = \Alg \cl L$.
Using Theorem \ref{4elgen3minus}, we have
$$\Lat(\cl B \otimes \cl A) = \Lat\Alg (\cl L\otimes\cl M) =
(\Lat\Alg \cl L)\otimes\cl M = (\Lat \cl B)\otimes(\Lat \cl A).$$
\end{proof}

\begin{corollary}\label{c_reff}
Let $\mathcal{M}$ be an ABSL having the rank one density property. A
subspace lattice $\cl L$ is reflexive if and only if $\cl L\otimes
\cl M$ is reflexive.
\end{corollary}
\begin{proof} If $\cl L$ is reflexive then $\cl L\otimes\cl M$ is
reflexive by Theorem \ref{4elgen3minus}. Conversely, suppose that
$\cl L\otimes\cl M$ is reflexive. Let $L\in \Lat\Alg\cl L$ and $E\in
\cl M$ be an atom. By Theorem \ref{4elgen3minus}, $L\otimes E\in \cl
L\otimes\cl M$ and, by Corollary \ref{c_cup2}, $L\in \cl L$.
\end{proof}

\begin{corollary}\label{thesis0ncor}
If $\mathcal{L}$ is a subspace lattice having property (p) and
$\mathcal{M}$ is an ABSL having the rank one density
property, then $\cl L \otimes \cl M$ has property (p).
\end{corollary}
\begin{proof}
By hypothesis, we have that $\cl P \otimes \cl L$ is reflexive. It
follows from Corollary \ref{4elgen4} that $\cl P \otimes \cl L
\otimes \cl M$ is reflexive, that is, $\cl L\otimes\cl M$ has property (p).
\end{proof}

\begin{corollary}\label{4elgen3}
Let $H$ be a Hilbert space and $P$ and $Q$ be projections acting on
$H$ such that $P \wedge Q = 0$ and $P \vee Q = I$. If
$\mathcal{M}=\{0,P,Q,I\}$ and $\mathcal{L}$ is a subspace lattice
acting on a Hilbert space $K$, then
$$\Lat\Alg(\mathcal{L} \otimes \mathcal{M}) =
\{(L_1 \otimes P) \vee (L_2 \otimes Q): L_1,L_2 \in
\Lat\Alg\mathcal{L} \}.$$ Furthermore, if $\cl L$ is reflexive, then
the LTPF holds for $\Alg \cl L$ and $\Alg \cl M$, and
the lattice $\cl L\otimes\cl P$ is reflexive.
\end{corollary}
\begin{proof}
The statement is immediate from Theorem \ref{4elgen3minus},
Corollary \ref{4elgen4} and the fact that two atom ABSLs satisfy
the rank one density property \cite[Theorem 2.1]{Papadakis}.
\end{proof}

We finish this section with the following additional consequence of
the above results.

\begin{theorem}
Let $\mathcal{L}$ and $\mathcal{M}$ be ABSLs with sets of atoms $\{D_i : i
\in I\}$ and $\{E_j : j \in J\}$, respectively. If either $\cl L$ or $\cl M$ has
the rank one density property, then $\mathcal{L} \otimes
\mathcal{M}$ is an ABSL whose set of atoms is $\{D_i \otimes E_j : (i,j) \in I \times J\}$.
\end{theorem}
\begin{proof}
Without loss of generality, we assume that $\mathcal{M}$ has the
rank one density property. By Corollary \ref{4elgen4} and the fact
that every ABSL is reflexive \cite{boolHalmos}, we have that
$$\mathcal{L} \otimes \mathcal{M} = \{ \underset{j \in J}{\vee}(P_j
\otimes E_j): P_j \in \mathcal{L}, \ j \in J \}.$$
On the other hand, for $j\in J$, we have that
$$P_j \otimes E_j = (\vee_{D_i\leq P_j} D_i)\otimes E_j = \vee_{D_i\leq P_j} D_i \otimes E_j.$$
Thus, every element in $\mathcal{L} \otimes \mathcal{M}$ is the
span of elements of the set $\{D_i \otimes E_j :
(i,j) \in I \times J\}$.

Suppose that $L = \underset{j \in J}{\vee}(P_j \otimes E_j) \subsetneq (D_{i_0} \otimes E_{j_0})$ where
$(i_0,j_0) \in I \times J$. Since $\cl L$ is an ABSL, we have that
either $P_{j_0} \wedge D_{i_0} = 0$, or $D_i \subseteq P_{j_0}$. If $D_{i_0}
\subseteq P_{j_0}$, then $D_{i_0} \otimes E_{j_0} \subseteq L$ and thus $D_{i_0} \otimes E_{j_0} = L$.
By hypothesis, $D_{i_0} \otimes E_{j_0} \neq L$, hence $P_{j_0}
\wedge D_{i_0} = 0$. By Theorem \ref{4elang6},
$$L = L \wedge (D_{i_0} \otimes E_{j_0}) = (P_{j_0} \wedge D_{i_0}) \otimes E_{j_0} = 0.$$
Thus, $D_i\otimes E_j$ is an atom of $\cl L\otimes\cl M$ for all $i$ and $j$.

It remains to prove that $\mathcal{L} \otimes \mathcal{M}$ is
complemented and distributive. Let $L = \underset{j \in J}{\vee}(P_j
\otimes E_j)$, where $P_j \in \mathcal{L}$, $j \in J$, and let
$P'_j$ be the complement of $P_j$ in $\cl L$, for all $j \in J$. If $L' =
\underset{j \in J}{\vee}({P'}_j \otimes E_j)$, then
$$L \vee L' = \underset{j \in J}{\vee}((P_j \vee {P'}_j) \otimes E_j) = \underset{j \in J}{\vee}(I \otimes E_j) = I$$
and, by Theorem \ref{4elang6},
$$L \wedge L' = \underset{j \in J}{\vee}((P_j \wedge {P'}_j) \otimes E_j) = 0.$$
Hence $L'$ is a complement for $L$.
Finally, the distributivity of $\cl L\otimes\cl M$ follows from Corollary \ref{c_dist}.
\end{proof}

\medskip

\noindent {\bf Acknowledgements} We would like to thank A. Katavolos
and V.S. Shulman for their remarks which helped us improve the
exposition of the paper.

\end{document}